\documentclass{amsart}
\usepackage[a4paper,asymmetric]{geometry}
\usepackage{amsmath,amssymb,amsthm,amscd,bbm} 
\usepackage{graphicx,tikz,ifthen,verbatim,enumerate}
\usepackage[colorlinks,allcolors=blue,draft=false]{hyperref}
\usepackage[justification=centering]{caption}

\let\oldfn\footnote \renewcommand\footnote[1]{\oldfn{\,#1}}

\usepackage{cleveref} 
\newtheorem{thm}{Theorem}
\newtheorem*{thm*}{Theorem}
\newtheorem{stat}[thm]{Statement}

\newtheorem{lem}[thm]{Lemma}
\newtheorem{prop}[thm]{Proposition}

\theoremstyle{definition}
\newtheorem{defn}[thm]{Definition}
\newtheorem{rem}[thm]{Remark}
\newtheorem{que}{Question}
\newtheorem{example}{Example} 

\makeatletter\@ifundefined{proof}{\newenvironment{proof}[1][Proof.]{\noindent\textit{#1} }{\hfill$\square$\par}}{}\makeatother
\newcounter{commentcounter} 
\usepackage{silence} 
\def\M{{\mathcal M}}
\def\N{{\mathbb N}}
\def\Z{{\mathbb Z}}
\def\P{{\mathcal P}}

\def\1{{\mathbbm1}}
\def\pp{{\bar p}}
\def\qq{{\bar q}}
\def\1{{\mathbbm 1}}

\def \tl {topological}
\def \im {invariant measure}
\def \inv {invariant}
\def \ds {dynamical system}
\def \sq {sequence}

\def \eps {\varepsilon}
\providecommand\doublecyl[2]{{\big[\begin{smallmatrix}#1 \\ #2\end{smallmatrix}\big]}}

\title{Destruction of {CPE}-normality along deterministic sequences}
\author{Adam Abrams}
\address{Faculty of Pure and Applied Mathematics, Wroc\l{}aw University of Science and Technology, \linebreak{}Wroc\l{}aw, Poland}
\email{the.adam.abrams@gmail.com}
\author{Tomasz Downarowicz}
\address{Faculty of Pure and Applied Mathematics, Wroc\l{}aw University of Science and Technology, \linebreak{}Wroc\l{}aw, Poland}
\email{dowar@pwr.edu.pl}

\begin{document}
\maketitle

\begin{abstract}Let $\mu$ be a shift-invariant measure on $\Lambda^{\mathbb N}$, where $\Lambda$ is a finite or countable alphabet. We say that an infinite subset $S=\{s_1,s_2,\dots\}\subset\N$ (where
$s_1<s_2<\dots$) \emph{preserves (destroys) $\mu$-normality} if, for any $x=(x_1,x_2,\dots)\in\Lambda^{\mathbb N}$ generic for $\mu$, the sequence $x|_S=(x_{s_1},x_{s_2},\dots)$ is (is not) generic for $\mu$. It is known from Kamae and Weiss that if $\mu$ is i.i.d.~then any deterministic set of positive lower density preserves $\mu$-normality. We show that deterministic sets, except ones with a very primitive structure that we call ``superficial'', destroy $\mu$-normality for any non-i.i.d.~measure $\mu$ with completely positive entropy (CPE). This generalizes Heersink and Vandehey's result for arithmetic progressions and the Gauss measure (associated to the continued fraction transformation). We give several examples showing that outside the class of measures with CPE, $\mu$-normality preservation can coexist with nearly any combination of three parameters: determinism of $S$ (or its lack), entropy of $\mu$ (zero or positive), and disjointness (or its lack) between $\mu$ and the measures quasi-generated by~$\mathbbm1_S\in\{0,1\}^{\mathbb N}$.
\end{abstract}

\section{Introduction}

A \emph{normal number} to base $10$ is a number $x\in[0,1]$ for which every block of digits $B=(b_1,b_2,\dots,b_k)$, where $k\in\N$ and each $b_i\in\{0,1,\dots,9\}$, appears in the decimal expansion of $x$ with the limit frequency $10^{-k}$.
In the language of ergodic theory, $x$ is normal when it is generic\footnote{Definitions of all necessary terminology are provided in the next section.} for the Lebesgue measure under the map $x\mapsto 10x \bmod 1$. In terms of symbolic dynamics, a number $x$ is normal if and only if its decimal expansion, viewed as an element of the unilateral (i.e., one-sided) shift on $10$ symbols, is generic for the unilateral $(\frac1{10},\frac1{10},\dots,\frac1{10}$)-Bernoulli measure. In 1909, E.~Borel showed\footnote{Borel's proof raised many controversies among mathematicians. From today's perspective it is clear that formally it does have a gap. The Borel--Cantelli Lemma proved several years later (see~\cite{Ca}) fixes the gap. Today normality of almost every number follows immediately from the Birkhoff Ergodic Theorem.} that almost every number $x\in[0,1]$ is normal (see~\cite{Bo}).
In 1949, D.~D.~Wall~\cite{Wa} proved that if $0.b_1b_2b_3\dots$ is normal then $0.b_\ell b_{\ell+m}b_{\ell+2m}b_{\ell+3m}\dots$ is normal as well, regardless of the parameters $\ell,m\in\N$. This result can be verbalized as follows: 

\begin{stat} \emph{Arithmetic progressions \emph{preserve} (classical) normality.} \end{stat}

\providecommand\maybeemph[1]{#1}
This phenomenon was generalized in the early 1970s by Kamae and Weiss~\cite{We,Ka}, who showed that a set $S\subset\N$ of positive lower density preserves normality if and only if it is \maybeemph{deterministic} (which means that all measures \maybeemph{derived} from $S$ have entropy zero). We remark that the Kamae--Weiss Theorem remains valid, by the same proof, if the uniform Bernoulli measure is replaced by any discrete Bernoulli measure (i.e., an i.i.d.~process with finitely or countably many states). 
\smallskip

A number $x\in[0,1]$ is called \maybeemph{CF-normal} (CF stands for ``continued fraction'') if it is generic under the Gauss map for the Gauss measure on the interval. This is equivalent to the condition that every block $B=(b_1,b_2,\dots,b_k)$, where $k\in\N$ and each $b_i\in\N$, appears in the continued fraction expansion $[0;a_1,a_2,\dots]$ of $x$ with limit frequency equal to the Gauss measure of the cylinder $[B]$ (i.e., the subset of $[0,1]$ determined by $B$).\footnote{This sentence uses square brackets in three ways: a continued fraction, a cylinder, and the unit interval. Throughout the paper, the meaning of brackets should be clear from context (continued fraction expansions are also distinguished by the semicolon after $0$).}

A result obtained in 2016 by Heersink and Vandehey~\cite{HV} stands in drastic contrast to Wall's Theorem. It asserts that for any CF-normal number $x=[0;a_1,a_2,\dots]$ and any natural parameters $\ell,m$ with $m \ge 2$, the number $z=[0;a_\ell,a_{\ell+m},a_{\ell+2m},\dots]$ is never CF-normal. This is much stronger than the claim that arithmetic progressions do not preserve CF-normality (for that it would be sufficient to have one counterexample $(x,\ell,m)$ for which the resulting number $z$ is not CF-normal). To emphasize this distinction, we make the following verbalization: 

\begin{stat} \label{HV} \emph{Non-trivial arithmetic progressions \emph{destroy} CF-normality.} \end{stat}

The Heersink--Vandehey result might seem counterintuitive. In~\cite{We}, the key ingredient in Weiss' proof (that if a set $S\subset\N$ is deterministic then $S$ preserves classical normality) is the fact that unilateral Bernoulli shifts are disjoint (in the sense of Furstenberg) from all zero-entropy systems. Indeed, they have so-called \maybeemph{completely positive entropy (CPE)}, and H.~Furstenberg, in his seminal paper on disjointness~\cite{Fu}, proved that systems with CPE are disjoint from all zero-entropy systems. The Gauss system has CPE (this follows from the fact that it is a unilateral measure-theoretic factor of a K-automorphism, see~\cite[Corollary~2]{Na}), so it is disjoint from all zero-entropy systems as well. In other words, the Gauss system enjoys the key property needed for the validity of Weiss' proof. Based on disjointness between systems with CPE and zero-entropy systems, Kamae even formulated a remark (\cite[page 147]{Ka}) that says, roughly speaking, that $\mu$-normality (understood as being generic for a measure $\mu$ with CPE) is \emph{in some sense} preserved by essential deterministic sets of positive lower density. Did Heersink and Vandehey prove Kamae wrong?
The answer is no. Kamae's remark, precisely translated to a more contemporary language, states that if $\mu$ has CPE then deterministic sets of positive lower density preserve a slightly weaker property than $\mu$-normality that, following Borel's terminology, should be called \maybeemph{simple $\mu$-normality}. It means that frequencies of just single symbols are preserved. And this statement is true. In fact, it is an easy observation (see \Cref{main3}) that for any shift \im\ $\mu$ (including those of zero entropy), disjointness of $\mu$ from all measures derived from a set $S$ with positive lower density implies that $S$ preserves simple $\mu$-normality. But for ``full'' $\mu$-normality preservation, disjointness is not enough, and the Heersink--Vandehey result is an example of this fact.

This raises some natural questions:
\begin{enumerate}
	\item Exactly why does the argument of classical normality preservation fail for CF-normality? 
	\item How badly does it fail? That is, does it fail for all deterministic sets of positive lower density or just for arithmetic progressions? 
	\item How important is the fact that the Gauss system has CPE without being a Bernoulli shift?
	\item Does the analogue of the Kamae--Weiss Theorem, or at least of Wall's Theorem, still hold for systems that are isomorphic to unilateral Bernoulli shifts?
	\item Does the phenomenon of normality preservation exist for systems without CPE or with entropy zero?
	\item Is disjointness of $\mu$ from all measures derived from $S$ necessary for $S$ to preserve $\mu$-normality? 
\end{enumerate}

In this paper we answer all these questions, in particular revealing a surprising peculiarity of Bernoulli shifts (in other words, of i.i.d.~processes). We will show that \emph{the phenomenon of normality preservation exists for no systems with CPE except for Bernoulli shifts} (see Theorems~\ref{main1} and~\ref{main2}). The property of Bernoulli shifts which other systems with CPE are lacking, and which (in addition to disjointness from zero-entropy systems) is crucial for $\mu$-normality preservation of deterministic sets, is ``spreadability'' (see \Cref{sec CPE not B}).


The question remains whether $\mu$-normality preservation is possible in the case where $\mu$ has positive entropy and $S$ is not deterministic. In this case, among measures derived from $S$ there exists at least one of positive entropy and this one is not disjoint from $\mu$.\footnote{Two measures of positive entropy are never disjoint. In fact, their natural extensions have a common Bernoulli factor (see~\cite{Si}).} But whether this implies a lack of $\mu$-normality preservation appears to be a very difficult problem. All we know is that, in view of \Cref{main2}, in any potential example $\mu$ must not have CPE. But even if it has, a similar question about simple $\mu$-normality preservation is open. We append a short section containing this and other open problems.

\section{Preliminaries} \label{sec:preliminaries}
\subsection{Generic and quasi-generic points} We consider a \tl\ \ds\ $(X,T)$, where $X$ is a compact metric space and $T:X\to X$ is a continuous transformation. By the \mbox{weak-star} compactness of the set of Borel probability measures on $X$, for every $x\in X$ the sequence of atomic measures 
\[
A_n(x) = \frac1n\sum_{i=0}^{n-1}\delta_{T^i(x)}
\]
has at least one accumulation point. If $\mu$ is such an accumulation point, it is necessarily $T$-invariant. We then say that $x$ \emph{quasi-generates} $\mu$ or that $x$ is \emph{quasi-generic} for $\mu$. The set of all measures quasi-generated by $x$ will be denoted by $\M_x$. If the entire \sq\ $A_n(x)$ converges to some measure $\mu$ (equivalently, if $\M_x=\{\mu\}$) then we say that $x$ \emph{generates} $\mu$ or that $x$ is \emph{generic} for $\mu$.  If $\mu$ is ergodic, then, by virtue of the Birkhoff Ergodic Theorem, $\mu$-almost every point is generic for $\mu$.

We will mostly focus on \emph{unilateral shift systems} $(\Lambda^\N,\sigma)$ on a finite\footnote{We always assume that $\#\Lambda\ge2$; otherwise $\Lambda^\N$ is a singleton and all our considerations become trivial.} or countable alphabet $\Lambda$, that is, for $x=(x_i)_{i\in\N}\in\Lambda^\N$, 
\[ 
    (\sigma\,x)_i = x_{i+1}, \quad i\in\N.
\]
For a finite block $B=(b_1,b_2,\dots,b_k)\in\Lambda^k$ ($k\in\N$) and $x\in\Lambda^\N$, the \emph{frequency} of $B$ in $x$ is defined as
\[
    \mathsf{Fr}_x(B) = \lim_{n\to\infty}\frac1n\#\big\{i\in\{1,2,\dots,n\}: (x_i,x_{i+1},\dots,x_{i+k-1}) =B\big\},
\]
provided the limit exists.
It is elementary to see that an element $x$ of a shift system is generic for a shift-\im~$\mu$ if and only if $\mathsf{Fr}_x(B)$ equals $\mu([B])$ for all finite blocks $B$, where \[ [B]=\{y\in X: (y_1,\dots,y_k) =B\} \] is called the \emph{cylinder} associated with $B$. It is well known that for any \im\ $\mu$ on the symbolic space $\Lambda^\N$ there exists an element $x\in\Lambda^\N$ generic for $\mu$.

We remark that all of these definitions extend naturally to bilateral shifts on $\Lambda^{\mathbb Z}$, with frequency defined either in a two-sided or one-sided manner (leading to two different meanings of~generic points). The notion of generic points for measures also extends to $\Lambda^G$ with $G$ a count\-able amenable group or cancellative semigroup, although the definition of frequency is more involved in this setting (see~\cite[Section~5.3]{BDV}).

\subsection{Bernoulli shifts and the Gauss system}
Let $P=\{p_a:a\in\Lambda\}$ be a probability vector assigning probabilities to the symbols of the alphabet $\Lambda$. The product measure $\mu_P$ is determined by the following equalities: for each $k \in \N$ and each block $B=(b_1,b_2,\dots,b_k)\in\Lambda^k$ we have
\[
\mu_P([B])=p_{b_1}p_{b_2}\cdots p_{b_k}. 
\]
Note that if we treat the symbols appearing at the consecutive coordinates as $\Lambda$-valued random variables then the product measure $\mu_P$ corresponds to the case where the above random variables are independent and identically distributed (i.i.d.). In ergodic theory, the measure $\mu_P$ is called \emph{the Bernoulli measure} (associated to $P$), and the system $(\Lambda^\N,\mu_P,\sigma)$ is called a \emph{Bernoulli shift}. 
If $\Lambda$ is finite and $P$ is uniform (that is, $p_a=\frac1{\#\Lambda}$ for each $a\in\Lambda$) then $\mu_P$ is called the \emph{uniform Bernoulli measure} and each point generic for $\mu_P$ is traditionally called \emph{normal} in base $\#\Lambda$. By analogy to this traditional notion of normality, we will say that $x\in\Lambda^\N$ is \emph{$\mu$-normal} if it is generic for some shift-\im~$\mu$ (not necessarily Bernoulli).
\medskip

In particular, we will consider the Gauss system $([0,1],\lambda,G)$, where 
\[ 
G(x) = \left\{\frac1x\right\} \text{ ($\{\cdot\}$ denotes the fractional part function)}, \ \ {\rm d}\lambda=\frac{{\rm d}x}{(x+1)\log2}. 
\] 
If numbers $x\in[0,1]$ are represented by their continued fraction expansions
\[
\quad x=[0;a_1,a_2,\dots]=\smash{ \frac1{a_1+\frac1{a_2+\frac1{a_3+\cdots}}} } \phantom{\frac9{9_y}}
\]
then the Gauss transformation is simply the shift:
\[
G([0;a_1,a_2,a_3\dots])=[0;a_2,a_3,\dots].
\]
Then the Gauss measure $\lambda$ on $[0,1]$ corresponds to some shift-\im\ on $\N^\N$. By a slight abuse of terminology, we will call this measure on $\N^\N$ the \emph{Gauss measure} as well. Points in~$[0,1]$ (or elements of $\N^\N$ via their continued fraction expansions) that are generic for the Gauss measure will be referred to as \emph{CF-normal}.

\subsection{Entropy, CPE, and determinism} \label{ss CPE} Let $\mu$ be a probability measure on some standard measurable space $X$, and let~$\P$ be a finite measurable partition of~$X$. The Shannon entropy of~$\P$ is defined as
\[
H_\mu(\P)=-\sum_{A\in\P}\mu(A)\log(\mu(A)).
\]
For $(X,\mu,T)$ a measure-preserving system, the \emph{entropy of the process generated by $\P$} is
\[
h_\mu(\P) = \lim_{n\to\infty}\frac1nH_\mu \!\left( \bigvee_{i=0}^{n-1}T^{-i}\P \right),
\]
and the \emph{Kolmogorov--Sinai entropy} of the system $(X,\mu,T)$ is 
\[
h_\mu(T) = \sup_\P h_\mu(\P),
\]
where the supremum is taken over all finite measurable partitions of $X$.
Often the transformation~$T$ will be fixed (for example, as the shift map $\sigma$), in which case we may write~$h(\mu)$ in place of~$h_\mu(T)$ and call it briefly \emph{the entropy of~$\mu$}. This convention allows us to speak about ``{measures of positive entropy}'' and ``{measures of entropy zero}.''

A system (measure) is said to have \emph{completely positive entropy (CPE)} if for every nontrivial finite partition $\P$ of $X$ the entropy $h_\mu(\P)$ is positive. It is well known that Bernoulli shifts have CPE. Another important example of a system with CPE is the Gauss system.

We remark that the class of invertible systems with CPE coincides with the class of \emph{Kolmogorov automorphisms}, commonly abbreviated as K-automorphisms (we skip the original definition; see, for example,~\cite{Parry}, which also contains the characterization via CPE originally due to Kolmogorov and Sinai). Some authors use the name ``K-systems'' for (not necessarily invertible) systems with~CPE.  
\medskip

At the extreme opposite to CPE is the notion of determinism. According to our needs, we will discuss this property exclusively in the context of infinite subsets of $\N$. If we identify $S \subset \N$ with its characteristic function $y=\1_S$ (from now on the letter $y$ is reserved to denote $\1_S$), it becomes an element of the binary symbolic space $\{0,1\}^\N$. A shift-invariant measure $\nu$ on this space is said to be \emph{derived from~$S$} if $\nu\in\M_y$, and the set $S$ is called \emph{deterministic} if all measures derived from~$S$ have zero entropy.

\subsection{Mixing of all orders}
A measure-preserving system $(X,\mu,T)$ is \emph{mixing} if for any pair of measurable sets $A,B$ one has
\[
\lim_{n\to\infty}\mu(A\cap T^{-n}B) = \mu(A)\mu(B).
\]
The system is \emph{mixing of all orders} if for any finite collection of sets $A_0,A_1,A_2,\dots,A_k$ ($k\in\N$) one has
\[
\lim_{n_1,\dots,n_k\to\infty}\mu(A_0\cap T^{-n_1}A_1\cap T^{-n_1-n_2}A_2\cap\cdots\cap T^{-n_1-n_2-\cdots-n_k}A_k) = \mu(A_1)\mu(A_2)\cdots\mu(A_k).
\]
It is a long standing open question, posed by Rokhlin in~\cite{R49}, whether mixing implies mixing of all orders. We know, however, that completely positive entropy does imply mixing of all orders~\cite{R61}.

\subsection{Preservation and destruction of normality}
Let $\Lambda^\N$ be a unilateral symbolic space system on a finite or countable alphabet $\Lambda$ with a fixed shift-\inv\ probability measure $\mu$.
An infinite set $S = \{s_1,s_2,\dots\}$ of natural numbers (equivalently, an increasing sequence in $\N$) is said to \emph{preserve $\mu$-normality} if for any $\mu$-normal element $x=(x_1,x_2,\dots)\in\Lambda^\N$ the element
\[
    x|_S = (x_{s_1},x_{s_2},\dots)\in\Lambda^\N
\] 
is also $\mu$-normal. We say that $S$ \emph{destroys $\mu$-normality} if for any $\mu$-normal~$x$ the element $x|_S$ is not $\mu$-normal. We emphasise that in both cases the requirement involves all (not only $\mu$-almost all) $\mu$-normal elements $x$. The weaker notion of ``almost sure $\mu$-normality preservation'' is much more easily satisfied (for example, any infinite set $S$ preserves classical normality almost surely) and will not be addressed in this paper.

\subsection{Simple normality}
Borel in~\cite{Bo} calls a number $x\in[0,1]$ \emph{simply normal} in base 10 if every digit $0, 1, 2, \dots, 9$ appears in the decimal expansion of $x$  with the frequency $\frac1{10}$. Following this terminology, we will call an element $x\in\Lambda^\N$ \emph{simply $\mu$-normal} if for every symbol $a\in\Lambda$ one has
\[
\mathsf{Fr}_x(a) = \mu([a]).
\]

We say that an infinite set $S\subset\N$ \emph{preserves simple $\mu$-normality} if for any element $x\in \Lambda^\N$ that is $\mu$-normal (attention! not just simply $\mu$-normal) the element $x|_S$ is simply $\mu$-normal. In this situation, both Kamae and Weiss use a different terminology. If $x\in\Lambda^\N$ and $S\subset\N$ are such that
\[
\mathsf{Fr}_x(a) = \mathsf{Fr}_{x|_S}(a)
\]
for all $a\in\Lambda$, they say that $x$ is an ``$S$-collective.'' In this language, $S$ preserves simple $\mu$-normality if every $\mu$-normal sequence is an $S$-collective.

\subsection{Lower and upper density}
The lower density and upper density of a set $S\subset \N$ are defined, respectively, as
\begin{align}\label{uld}
\begin{split}
\underline d(S)&=\liminf_{n\to\infty} \frac{\#(S\cap\{1,\dots,n\})}{n}, \\
\overline d(S)&=\limsup_{n\to\infty} \frac{\#(S\cap\{1,\dots,n\})}{n}.
\end{split}
\end{align}
If the two are equal, we speak about density of $S$. Clearly, ``lower density $1$'' is synonymous with ``density $1$'' and ``upper density $0$'' is synonymous with  ``density $0$.'' By compactness of the interval $[0,1]$, the limit in \eqref{uld} exists along an increasing \sq\ $(n_k)_{k\ge1}$. Then we speak about $(n_k)$-density of $S$.

\subsection{Joinings and disjointness} 
Let $(X,\mu,T)$ and $(Z,\nu,U)$ be two measure-preserving systems. 
A~\emph{joining} of~$\mu$ and~$\nu$, denoted by~$\mu\vee\nu$, is any measure $\xi$ on the product space $X\times Z$ with the following properties:
\begin{itemize}
\item the marginal of $\xi$ on $X$ is $\mu$ (that is, $\xi(A \times Z) = \mu(A)$ for all measurable $A \subset X$),
\item the marginal of $\xi$ on $Z$ is $\nu$, and
\item $\xi$ is \inv\ under the product map $T\times U$ defined by $(T\times U)(x,z)=(Tx,Uz)$.
\end{itemize}
There always exist at least one joining, namely the {product measure} $\mu\times\nu$. If $\mu$ and $\nu$ admit no joinings other than the product measure, then they are called \emph{disjoint}. This notion was introduced by~H.~Furstenberg in his seminal paper~\cite{Fu}, in which he proves that Bernoulli shifts are disjoint from all zero-entropy systems. The truth is that the class of systems that are disjoint from all zero-entropy systems coincides with CPE. Furstenberg refers in this aspect to Rokhlin~\cite{R60}, who in turn attributes the result (without reference) to Pinsker.
\medskip

The following observation is useful. Suppose that $x\in X$ is generic for an \im~$\mu$, and let $z\in Z$. Then any measure on $X\times Z$ quasi-generated by the pair $(x,z)$ is a joining of $\mu$ with some \im~$\nu$ on $Z$ that is quasi-generated by $z$. 

\subsection{Known results}
We summarize here the existing results about $\mu$-normality preservation and destruction:
\begin{itemize}
    \item Arithmetic progressions preserve classical normality (\cite{Wa}).
    \item Deterministic sets with positive lower density preserve classical normality (\cite{We}).
    \item Among sets with positive lower density, only deterministic ones preserve classical normality (\cite{Ka}).
    \item If $\mu$ has CPE then deterministic sets with positive lower density preserve simple $\mu$-normality (\cite{Ka}).
    \item Nontrivial arithmetic progressions destroy CF-normality (\cite{HV}).
\end{itemize}
We remark that for actions of countable amenable groups (or cancellative semigroups) $G$ and a uniform Bernoulli measure $\mu$ on $\Lambda^G$, the paper~\cite{BDV} provides a characterization of $\mu$-normality-preserving sets $S\subset G$ as precisely the deterministic sets with positive lower density. It has to be noted, however, that $\mu$-normality of an element $x\in\Lambda^G$ along a set $S\subset G$ is defined in a slightly different way and does not correspond to $\mu$-normality of $x|_S$ if $G=\N$ or $G=\Z$. With the definition of $\mu$-normality of $x$ along $S$ taken from~\cite{BDV}, the main result of this paper (\Cref{main1}) becomes false. So far, a notion of $\mu$-normality of $x$ along a subset $S\subset G$ that is applicable to actions of more general countable (semi)groups $G$ and that would lead to results fully analogous as for~$\N$ or~$\Z$ has not been found.

\section{Our results}

\subsection{Superficial sets}
As we are interested in $\mu$-normality destruction, we first need to deal with some classes of sets
which never destroy $\mu$-normality for trivial reasons associated with their structure. We call these sets ``superficial''. The name is meant to reflect the uncomplicated architecture of these sets. The main theorem of the paper will be concerned with non-superficial sets. The formal definition follows:

\begin{defn} \label{def:superficial}
An infinite set $S\subset\N$ is \emph{superficial} if, for $y=\mathbbm 1_S$, $\M_y$ is contained in the convex hull $\mathsf{conv}\{\delta_{\bar0},\delta_{\bar1}\}$, where $\delta_{\bar0}$ and $\delta_{\bar1}$ are the measures concentrated at the constant sequences $(0,0,...)$ and $(1,1,...)$, respectively.
\end{defn}

\begin{rem}Superficiality does not imply that the orbit closure of $y$ supports only two ergodic measures, $\delta_{\bar0}$ and $\delta_{\bar1}$. Along a set of density zero $y$ may contain arbitrarily long blocks of arbitrary structure, in particular, its orbit closure may even be the full shift on two symbols.
\end{rem}

\pagebreak[4]
\begin{prop}\label{superf}
An infinite set $S\subset\N$ is superficial if and only if $\N$ splits into three disjoint parts: $\mathbb A,\mathbb B,\mathbb C$, where
\begin{itemize}
    \item $\mathbb A$ (possibly empty) has density zero,
    \item $\mathbb B$ is either empty or is a union of infinitely many intervals whose lengths tend to infinity, and $\mathbb B\subset S$ (i.e., $y|_{\mathbb B}$ consists only of $1$'s),
    \item $\mathbb C$ is either empty or is a union of infinitely many intervals whose lengths tend to infinity, and $\mathbb C\cap S=\emptyset$ (i.e., $y|_{\mathbb C}$ consists only of $0$'s).
\end{itemize}
\end{prop}

\begin{proof}
It is obvious that a shift-\im\ $\nu$ on $\{0,1\}^\N$ belongs to $\mathsf{conv}\{\delta_{\bar0},\delta_{\bar1}\}$ if and only if the cylinder associated to any block $B$ containing two different symbols has measure~$0$. 

If $S$ is such that $\N$ splits into $\mathbb A,\mathbb B,\mathbb C$, as described above, then any block $B$ containing two different symbols occurs in $y$ either at places where an interval from $\mathbb A$ is adjacent to an interval from $\mathbb B$ or vice versa, or so that its domain intersects $\mathbb A$. All these occurrences jointly occupy a set of density zero, so $\nu([B])=0$ for any measure $\nu\in\M_y$, hence $\M_y\subset\mathsf{conv}\{\delta_{\bar0},\delta_{\bar1}\}$ and $S$ is superficial.

If $S$ is superficial then the cylinder associated to any \emph{wrapped constant block} $B$ (i.e., a block of the form $1000\dots01$ or $0111\dots10$) of finite length $m$ has measure zero for any $\nu\in\M_y$, implying that $B$ appears in $y=\1_S$ with density zero. Let $n_0=1$ and for $m\ge1$ let $n_m > n_{m-1}$ be an integer not contained in any wrapped constant block of length $m-1$ or $m$, such that for any $n\ge n_m$ the wrapped constant blocks of lengths less than or equal to $m$ occupy in $y|_{[1,n]}$ a set of cardinality less than $\frac nm$. Let $\mathbb A$ be the union of the domains of all wrapped constant blocks~$B$ in $y$ which occur after the coordinate $n_m$ and before the coordinate $n_{m+1}$, where $m$ is the length of~$B$. The set $\mathbb A$ has density $0$ (for any $n$, $y|_{[1,n]} \cap \mathbb A$ consists of blocks of length at most $m$, where $n_m \le n < n_{m+1}$, which appear with frequency less than $\frac 1m$), and its complement splits into $\mathbb B$ and $\mathbb C$ as in the formulation of the proposition.
\end{proof}

\begin{thm}\label{nogood}
Let $\mu$ be any \im\ on a symbolic space $\Lambda^\N$. If an infinite set $S=\{s_1,s_2,\dots\}\subset\N$ is superficial then it does not destroy $\mu$-normality.
\end{thm}
\begin{proof}
We need to construct a $\mu$-normal element $x\in\Lambda^\N$ such that $x|_S$ is also $\mu$-normal. We will proceed ``backwards'', i.e., we will first define $x|_S$ and then extend it to an $x\in\Lambda^\N$.
We start by choosing a $\mu$-normal element $z\in\Lambda^\N$ (we know that such an element exists).
According to Proposition~\ref{superf}, $\N$ splits into $\mathbb A,\mathbb B$ and $\mathbb C$.

If $\mathbb B$ is empty then $S\subset\mathbb A$. In this case we define $x|_S$ by letting $x_{s_i}=z_i$ for all $i\in\N$. Clearly, $x|_S=z$ so it is $\mu$-normal. On the complement of $S$ we let $x_i=z_i$. Since $x$ differs from $z$ only on the set $S$ of density zero, $x$ is $\mu$-normal. 

If $\mathbb B$ is nonempty then it is a disjoint union $\mathbb B=\bigcup_{n\ge1} I_n$, where each $I_n$ is an interval and the lengths $|I_n|$ tend to infinity. Let $b_n$ denote the first element in $I_n$. The set $S$
splits as $S=\bigcup_{n\ge1}(S\cap[b_n,b_{n+1}-1])$, where $S\cap[b_n,b_{n+1}-1]$ consists of $I_n$ possibly followed by some elements of $S\cap\mathbb A$. For each $n$ let us enumerate $S\cap[b_n,b_{n+1}-1]$ as $\{s_{n,1},s_{n,2},\dots,s_{n,i_n}\}$. We are in a position to define $x|_S$ by letting, for each $n\ge1$ and $j\in[1,i_n]$,
$$
x_{s_{n,j}}=z_j.
$$
In this manner, $x|_S$ is built of the initial blocks $z|_{[1,i_n]}$ where $i_n$ tends to infinity. It is elementary to see that, due to this structure, $x|_S$ is $\mu$-normal. 

In remains to define $x$ on the complement $S^c$ of $S$. This complement consists of $\mathbb C$ and $S^c\cap\mathbb A$. If $\mathbb C$ is empty then $S^c$ has density zero and $x$ is $\mu$-normal regardless of how we define it on $S^c$. Otherwise we define $x|_{S_c}$ similarly as $x|_S$: the set $\mathbb C$ is a disjoint union $\mathbb C=\bigcup_{n\ge1} J_n$, where each $J_n$ is an interval and the lengths $|J_n|$ tend to infinity. For each $n$ let us enumerate $J_n$ as $\{t_{n,1},t_{n,2},\dots,t_{n,l_n}\}$. We define $x|_{\mathbb C}$ by letting, for each $n\ge1$ and $j\in[1,l_n]$,
$$
x_{t_{n,j}}=z_j.
$$
On the set $S^c\cap\mathbb A$ we define $x$ arbitrarily. Note that $x|_{\mathbb B\cup\mathbb C}$ is built of initial blocks of $z$ whose lengths tend to infinity and thus it is $\mu$-normal. Since the remaining set $\mathbb A$ has density zero, $x$ is $\mu$-normal as well. 
\end{proof}

For completeness, we also discuss $\mu$-normality preservation of superficial sets. 
In this aspect, superficial sets split into two differently behaving subclasses.
\begin{enumerate}[(a)]
    \item Superficial sets with lower density zero, i.e., such that $\delta_{\bar 0}\in\M_y$;
    \item Superficial sets with positive lower density, i.e., such that $\delta_{\bar 0}\notin\M_y$.
\end{enumerate}

Sets from subclass (a) do not preserve $\mu$-normality for any \im\ $\mu$, which is a particular case of the following more general fact:

\begin{thm}\label{ojojoj}
Any infinite set $S\subset\N$ of lower density $0$ does not preserve (simple) $\mu$-normality for any \im\ $\mu$.
\end{thm}
\begin{proof}
First observe that if $S$ has density zero then we can choose any $\mu$-normal element $x$ and change it arbitrarily along $S$. The modified element $x'$ remains $\mu$-normal while $x'|_{S}$ can be any symbolic element, in particular one that is not simply $\mu$-normal.

Suppose now that $S\subset\N$ has lower density $0$ and positive upper density. By~\cite[~Lemma 5.4]{BDV} (applied to the classical F\o lner \sq\ in $\N$, $F_n=\{1,2,\dots,n\}$), there exists a subset $S'\subset S$ which, on the one hand, has density $0$, but on the other hand satisfies 
\[
\limsup_{n\to\infty}\frac{\#(S'\cap\{1,2,\dots,n\})}{\#(S\cap\{1,2,\dots,n\})}\ge\frac23.
\]
Now, we choose any $\mu$-normal element $x$ and we replace all symbols $x_s$, with $s\in S'$, by some constant symbol $a\in\Lambda$ such that $\mu([a])\le\frac12$ (clearly, such a symbol exists). This modification remains $\mu$-normal, while the upper density of the set where $a$ appears in $x|_S$ is at least $\frac23$ and so $x|_S$ is not simply $\mu$-normal. 
\end{proof}

\begin{thm}\label{jeden}
Any superficial set $S=\{s_1,s_2,\dots\}\subset\N$ with positive lower density (i.e., from the class (b) above) preserves $\mu$-normality for any ergodic measure $\mu$. 
\end{thm}

\begin{proof} Let $\mu$ be an ergodic measure on $\Lambda^\N$, let $x\in\Lambda^\N$ be $\mu$-normal, and let $S$ be superficial with positive lower density. We need to show that $x|_S$ is also generic for $\mu$. Denote, as usual, $y=\mathbbm1_S$.

Let $(n_k)_{k\ge1}$ be an increasing \sq\ along which $x|_S$ quasi-generates some \im~$\nu$. Note that $(n_k)_{k\ge1}$ refers to the enumeration with respect to to $S$, while the corresponding \sq\ within $\N$ is $(s_{n_k})_{k\ge1}$. 
We will show that $\nu=\mu$. By passing to a sub\sq\ we can assume that along $(s_{n_k})_{k\ge1}$, $y$ also quasi-generates an \im. By assumption, this measure has the form $\gamma\delta_{\bar0}+(1-\gamma)\delta_{\bar1}$, where $\gamma\in[0,1)$ ($\gamma<1$ follows from positive lower density of $S$). By Proposition~\ref{superf}, $\N$ splits into three parts $\mathbb A,\mathbb B$ and $\mathbb C$, where $\mathbb A$ has density zero while $\mathbb B$ and $\mathbb C$ consist of longer and longer intervals on which $y$ is constant. Note that $1-\gamma$ equals the $(s_{n_k})$-density of $S$, which also equals the $(s_{n_k})$-density of $\mathbb B\subset S$. It is now crucial (and easy to see) that since $\mathbb B$ has positive $(s_{n_k})$-density while $\mathbb A$ has $(s_{n_k})$-density zero, the set $S\cap \mathbb A$ has $(n_k)$-density zero (relative with respect to $S$). As a consequence, $x|_{\mathbb B}$ generates, along $(n_k)_{k\ge1}$, the same measure as does $x|_S$, i.e., the measure $\nu$. Note that $\mathbb C$ has the complementary $(s_{n_k})$-density~$\gamma$. If $\gamma=0$ then we let $\nu'$ be any \im, otherwise, by passing to a sub\sq\ once again, we can assume that $x|_{\mathbb C}$ generates, along the sub\sq\ $(m_k)_{k\ge1}$ of $\mathbb C$ with $m_k = \#\{ c \in \mathbb C : c \le s_{n_k} \}$, the same \im\ $\nu'$. It is now an elementary observation that in either case 
$$
\mu = (1-\gamma)\nu+\gamma\nu',
$$ 
with $1-\gamma>0$. Because $\mu$ is ergodic, it is an extreme point of the set of \im s, so we have $\nu=\mu$ as needed.
\end{proof}

To summarize, superficial sets of class (a) neither preserve nor destroy $\mu$-normality for any \im\ $\mu$, while superficial sets of class (b) preserve (hence do not destroy) $\mu$-normality for any ergodic measure $\mu$. This settles completely the question about $\mu$-normality preservation/destruction for superficial sets and there is no point in including such sets in our further (more interesting) discussion.

\subsection{Main results}
We prove three main theorems. The first two both involve systems with~CPE. The first one is about deterministic sets and is a vast generalization of the Heersing--Vandehey result:
\begin{thm} \label{main1}
	Let $\mu$ be a shift-\im\ on a symbolic space $\Lambda^\N$ that has CPE but is not a Bernoulli 
shift (i.e., not i.i.d.). 
        Any infinite non-superficial deterministic set $S$ destroys $\mu$-normality.
\end{thm}
Note that the assumptions of \Cref{main1} include measures that are isomorphic to Bernoulli shifts but are not ``genuine'' Bernoulli shifts (a system equipped with such a measure is often referred to as a \emph{Bernoulli system}).
The property which distinguishes Bernoulli shifts from other systems with CPE (and in fact from all other ergodic systems) is \emph{spreadability}, which means, roughly speaking, that the measure of a block remains fixed when the block is ``spread'' (the rigorous definition is provided in \Cref{sec CPE not B}). This property is irrelevant when simple $\mu$-normality is concerned, as blocks of length~$1$ cannot be spread. This is why simple $\mu$-normality is preserved in many cases when $\mu$-normality is not preserved (or is even destroyed).

\medskip
The next theorem is about non-deterministic sets (note that any non-deterministic set is non-superficial):
\begin{thm} \label{main2}
   Let $\mu$ be a shift-\im\ on a symbolic space $\Lambda^\N$ that has CPE. Any non-deterministic set does not preserve $\mu$-normality.
\end{thm}
The last main theorem deals with simple $\mu$-normality preservation:
\begin{thm} \label{main3}
     Let $S$ be a set of positive lower density. If a shift-\im\ $\mu$ on a symbolic space $\Lambda^\N$ is disjoint from all measures derived from $S$, then $S$ preserves simple $\mu$-normality.
\end{thm}
In view of Theorem \ref{ojojoj}, the positive lower density assumption is necessary (it does not follow automatically from the existence of a measure disjoint from all measures in $\M_y$).

\section{Deterministic sets and CF-normality} \label{sec destroying CF normality}

Before we prove the main result \Cref{main1} in full generality, i.e., for systems with CPE that are not Bernoulli shifts, we show that Heersink and Vandehey have practically provided us with all that is needed to prove that deterministic sets destroy CF-normality.
The proof uses two key facts. The first one is disjointness between the Gauss system and any system of entropy zero. The second one is that the strict inequality $\lambda([1,1])>\lambda([1,*^n,1])$ holds for all $n \ge 1$, where~$\lambda$~is the Gauss measure and $[1,*^n,1]$ denotes, for $n\ge 0$, the set of all numbers whose continued fraction expansions start with~$1$ and have~$1$ as the $(n+2)^\text{nd}$ digit. This second fact is exactly~\cite[Lemma~3.1]{HV}, and we take it for granted.
\begin{thm}\label{distcf}
 Non-superficial deterministic sets destroy CF-normality. That is, if the number $x = [0;a_1,a_2,\dots]$ is CF-normal and $S = \{s_1,s_2,\dots\}$ is a non-superficial deterministic set, then $x|_S = [0; a_{s_1}, a_{s_2}, a_{s_3}, \dots]$ is \emph{not} CF-normal.
\end{thm}

\begin{proof}
Let $S$ be a non-superficial deterministic set. By assumption, there exists an increasing \sq\ $(N_k)_{k\ge1}$ along which $y=\mathbbm1_S$ quasi-generates a measure $\nu\notin\mathsf{conv}\{\delta_{\bar 0},\delta_{\bar 1}\}$ (in particular, \mbox{$\nu([1])>0$}). 

Let $x=[0;a_1,a_2,\dots]$ be CF-normal. By convention, $x$ will also denote the sequence of continued fraction digits: $x=(a_1,a_2,\dots) \in \N^\N$. Consider the ``double sequence'' 
\[
\binom x{y} = \begin{pmatrix}a_1, & a_2, & a_3, & \dots \\ \1_S(1), & \1_S(2), & \1_S(3),& \dots \end{pmatrix},
\] 
which we view as one sequence whose entries are two-element columns.

There exists an increasing subsequence $(N_{k_i})_{i\ge1}$ of $(N_k)_{k\ge1}$ along which the double sequence $\binom x{y}$ quasi-generates a shift-invariant measure $\xi$ on the space of double sequences, $\N^\N\times\{0,1\}^\N$. Note that $\xi$ is a joining of the Gauss measure $\lambda$ with $\nu$. As we know, $\lambda$ has CPE, while (since $S$ is deterministic) $\nu$ is a measure of entropy zero. Therefore $\lambda$ and $\nu$ are disjoint, which implies that $\xi$ is exactly the product measure $\lambda\times\nu$. 

Observe that the block $(1,1)$ occurs in $x|_S$ if and only if a double block of the form 
\[
B_n=\begin{pmatrix} 1, & *^n, & 1 \\ 1, & 0^n, & 1 \end{pmatrix}
\]
occurs in $\binom x{y}$ for some $n\ge 0$.
Pick a large integer $N$, and let $C_N$ denote the number of occurrences of $(1,1)$ in $x|_S$ up to the position $N$ in $x$. Let $D_N=\#(S\cap\{1,2,\dots,N\})$ be the number of occurrences of $1$ in $y$ up to the position $N$. If $x|_S$ were CF-normal, we would have
\[
\lim_{N\to \infty}\frac{C_N}{D_N}= \lambda([1,1]).
\]
However, we will show that the lower limit on the left is strictly smaller.

Denote by $C_{n,N}$ the number of occurrences of the double block $B_n$ in $\binom x{y}$ up to position~$N$, and for $\eps > 0$ denote $n_\eps = \lfloor 1/\eps \rfloor$. By construction, we have the following equality:
\[
C_N=\sum_{n=0}^\infty C_{n,N} = \sum_{n=0}^{n_\eps}C_{n,N} + \sum_{n>n_\eps}C_{n,N}.
\]
As a consequence,
\[
\liminf_{N\to\infty} \frac{C_N}{D_N}\le\Bigg(\lim_{k\to\infty}\frac {N_k}{D_{N_k}}\Bigg)\Bigg(\lim_{i\to\infty} \sum_{n=0}^{n_\eps} \frac{C_{n,N_{k_i}}}{N_{k_i}}+\limsup_{N\to\infty} \sum_{n>n_\eps} \frac{C_{n,N}}N\Bigg).
\]
The ratios $\frac {N_k}{D_{N_k}}$ converge to $\frac1{\nu([1])}$ (recall that $\nu([1])>0$), while for each $n\ge 0$ the ratios $\frac{C_{n,N}}N$ converge along $(N_{k_i})_{i\ge1}$ to~$\xi([B_n])$. The blocks $B_n$ with $n>n_\eps$ are longer than $\frac1\eps$, hence $\sum_{n>n_\eps} \frac{C_{n,N}}N \le\eps$ regardless of~$N$, which implies that $\limsup_{N\to\infty} \sum_{n>n_\eps} \frac{C_{n,N}}N \le \eps$. 

Since $\xi=\lambda\times\nu$, we have $\xi([B_n])=\lambda([1,*^n,1])\cdot\nu([1,0^n,1])$. Putting these facts together and letting $\eps\to0$, we conclude that
\[
\liminf_{N\to\infty} \frac{C_N}{D_N}\le\frac1{\nu([1])}\sum_{n=0}^\infty \lambda([1,*^n,1])\cdot\nu([1,0^n,1]) = 
\sum_{n=0}^\infty \lambda([1,*^n,1])\cdot c_n,
\]
where $c_n=\tfrac{\nu([1,0^n,1])}{\nu([1])}$ can be interpreted as the conditional measure $\nu([1,0^n,1]\:|\:[1])$, i.e., the conditional probability that if the first symbol (in the second row) is $1$ then it is followed by exactly $n$ symbols $0$ and then a $1$. Because $\nu([1])>0$, the event that the first symbol $1$ is followed by just $0$'s has (conditional) probability zero, implying that $\sum_{n=0}^\infty c_n = 1$. So, the right hand side is a convex combination of the numbers $\lambda([1,*^n,1])$ in which only the term with $n=0$ equals $\lambda([1,1])$ and, by \cite[Lemma~3.1]{HV}, all other are all strictly smaller. If $c_0$ were equal to $1$ then $\nu$-almost every element starting with $1$ would have $1$ at the second coordinate. In terms of the occurrences of $1$'s in $y$, this would mean that in a long block of the form $y|_{[1,N_{k_i}]}$, the majority of digits $1$ are followed by the digit $1$. This is possible only when $y|_{[1,N_{k_i}]}$ is built of long blocks of $1$'s and long blocks of $0$'s (and a negligible part of undetermined structure), implying that $\nu\in\mathsf{conv}\{\delta_{\bar 0},\delta_{\bar 1}\}$, a case which we have eliminated. So, $c_0<1$ and hence $\sum_{n=1}^\infty c_n>0$. Any convex combination of terms smaller than or equal to $\lambda([1,1])$, with nonzero contribution of terms strictly smaller, is strictly smaller than $\lambda([1,1])$. This ends the proof.
\end{proof}

\section{Systems with CPE that are not Bernoulli shifts}\label{sec CPE not B}

The proofs of \Cref{main1} and \Cref{main2} are similar to that of Theorem \ref{distcf}. The only difference is that in the general case of system with CPE that is not a Bernoulli shift, instead of the block $(1,1)$ we will use a possibly different block $B$ with the property that by inserting stars inside $B$, no matter how many and no matter where, we lower the measure of the corresponding cylinder.
In order to find such a block, we invoke the property that systems with CPE are mixing of all orders (see~\cite{R61}, or see~\cite[p.\,52]{Parry} for a textbook reference).

\smallskip
Before we proceed, we establish some terminology. Let $\Lambda$ be a finite or countable alphabet and choose a natural number $k\ge1$. Consider a block $B=(b_1,b_2,\dots,b_k)\in\Lambda^k$ and a $(k\!-\!1)$-dimensional vector of nonnegative integers $\pp=(p_1,p_2,\dots,p_{k-1})\in\N_0^{k-1}$. We denote 
\[
B^\pp=(b_1,*^{p_1},b_2,*^{p_2},\dots,b_{k-1},*^{p_{k-1}},b_k).
\]
Intuitively, this is $B$ ``spread apart'' by inserting stars according to $\pp$. The cylinder set $[B^\pp]$ associated with such a block is determined by fixing the following symbols:
$b_1$ at the coordinate~1, $b_2$ at the coordinate $1+p_1$, $b_3$ at the coordinate $2+p_1+p_2$, etc., until $b_k$ at the coordinate $k-1+p_1+p_2+\cdots+p_{k-1}$, and letting all other symbols be arbitrary. 

On $\N_0^{k-1}$ we introduce a partial order $\prec$ by the following rule: for $\pp=(p_1,p_2,\dots,p_{k-1})$ and $\qq=(q_1,q_2,\dots,q_{k-1})$ we write $\pp\prec\qq$ if for each $i=1,2,\dots,k-1$ we have $p_i\le q_i$ and for at least one $i$ we have strict inequality $p_i < q_i$.

\begin{lem} \label{lem spread block}
Let $\mu$ be a shift-\im\ on a symbolic space $\Lambda^\N$ ($\Lambda$ finite or countable) such that the measure-preserving system $(X,\mu,\sigma)$ is mixing of all orders but not a Bernoulli shift. Then there exists $k\ge2$, a block $B\in\Lambda^k$, and a vector $\pp_0 \in \N_0^{k-1}$ such that for all vectors $\qq\succ\pp_0$ we have $\mu([B^\qq]) < \mu([B^{\pp_0}])$.
\end{lem}

\begin{proof} We say that $\mu$ is ``$k$-Bernoulli'' if for any block $B\in\Lambda^k$ and any $(k\!-\!1)$-dimensional vector~$\pp$ one has
\[
\mu([B^\pp])=\Pi(B):=\mu([b_1])\mu([b_2])\cdots\mu([b_k]).
\]
Since the system is not a Bernoulli shift, there exists a natural $k$ such that the system is not $k$-Bernoulli. Let $k$ be the minimal such integer. Clearly, $k\ge 2$. There exists a block $B\in\Lambda^k$ and a $(k\!-\!1)$-dimensional vector $\pp\in\N_0^{k-1}$ such that $\mu([B^\pp])\neq\Pi(B)$. Either $\mu([B^\pp])>\Pi(B)$ or 
$\mu([B^\pp])<\Pi(B)$, and in the latter case some other block $C\in\Lambda^k$ satisfies $\mu([C^\pp])>\Pi(C)$. Then we rename $C$ as $B$ and we have $\mu([B^\pp])>\Pi(B)$ anyway. 

Denote $\varepsilon = \mu([B^\pp])-\Pi(B)>0$, consider the function $f: \N_0^{k-1} \to \mathbb R$ given by
\[
f(\qq)=\mu([B^\qq]),
\]
and let $Q$ be the set of vectors $\qq$ such that $f(\qq)\ge\Pi(B)+\varepsilon$. We know the set $Q$ is non-empty because it contains the vector $\pp$. In the following two paragraphs we will show (using mixing of all orders) that $Q$ is finite.

Suppose that $Q$ is infinite. Then $Q$ contains a sequence $\{\qq^{(n)}:n\ge 1\}$ strictly increasing with respect to the order $\prec$. There exists a non-empty subset of indices $I\subset\{1,2,\dots,k-1\}$ such that for every $i\in I$ the $i^\text{th}$ terms $q^{(n)}_i$ of $\qq^{(n)}$ grow to infinity with increasing $n$, while for $i\notin I$ the terms $q^{(n)}_i$ are eventually constant. By skipping finitely many initial terms, we may assume that they are constant throughout the sequence. 

For a given $n$, the blocks of stars $*^{q_i^{(n)}}$ with $i\in I$ separate in $B^{\pp^{(n)}}$ blocks of the form \smash{$B_j^{\bar r_j}$}, where $j=1,2,\dots,|I|+1$, and $B_j\in\Lambda^{k_j}$, and $\bar r_j\in \N_0^{k_j-1}$ for some $1\le k_j<k$ (if $k_j=1$ then $B_j^{\bar r_j}$ is just a single symbol from $\Lambda$). Note that for each $j$ both $B_j$ and $\bar r_j$ do not depend on $n$. On the other hand, the blocks $B_j^{\bar r_j}$ are separated by distances that increase with $n$. This fact, together with the mixing of all orders property of our system, imply that $f(\qq^{(n)})=\mu([B^{\qq^{(n)}}])$ tends to the product
\[
\prod_{j=1}^{|I|+1}\mu([B_j^{\bar r_j}]).
\]
Because, for each $j=1,2,\dots,|I|+1$, we have $k_j<k$, our system is $k_j$-Bernoulli, which implies that $\mu([B_j^{\bar r_j}])=\Pi(B_j)$. Eventually, we have shown that $f(\qq^{(n)})$ tends to $\prod_{j=1}^{|I|+1}\Pi(B_j)=\Pi(B)$. This is a contradiction with the definition of the set $Q$.

Having proved that $Q$ is finite, we know that there exists a (not necessarily unique)~$\prec$-maximal element in $Q$, which we denote by $\pp_0$. By the definition of $Q$, $f(\pp_0)\ge\Pi(B)+\varepsilon$, while for any $\qq\succ\pp_0$ we have $\qq\notin Q$, implying the desired inequality $f(\qq)<\Pi(B)+\varepsilon\le f(\pp_0)$.
\end{proof}

\begin{proof}[Proof of \Cref{main1}.]
Let $x=(x_1,x_2,\dots)\in\Lambda^\N$ be $\mu$-normal, and let $S=\{s_1,s_2,\dots\}$ be a non-superficial deterministic set. We need to show that $x|_S$ is not $\mu$-normal.

Let $y=\1_S$ and consider the double sequence
\[
\binom x{y} = \begin{pmatrix}x_1, & x_2, & x_3, & \dots \\ \1_S(1), & \1_S(2), & \1_S(3),& \dots \end{pmatrix}.
\] 
As in the proof of \Cref{distcf}, this double \sq\ quasi-generates a measure $\xi$ that is a joining of $\mu$ and some measure $\nu$ derived from $S$, $\nu\notin\mathsf{conv}\{\delta_{\bar 0},\delta_{\bar 1}\}$. Since our system has CPE and $S$ is deterministic, we have $\xi=\mu\times\nu$. 

As we assume that $\mu$ is not a Bernoulli shift, we can use \Cref{lem spread block}.
Let $B=(b_1,b_2,\dots,b_k)\in\Lambda^k$ and $\pp_0=(p_1,p_2,\dots,p_{k-1})\in\N_0^{k-1}$ be respectively the block and vector from that lemma. Set $n_1=1$ and for $i=2,3,\dots,k$ set
\[
n_i=i+\sum_{j=1}^{i-1}p_j
\]
(for each $i=1,2,\dots,k$, the number $n_i$ is the position of $b_i$ in $B^{\pp_0}$; in particular, $n_k$ is the length of $B^{\pp_0}$).
Observe that $B^{\pp_0}$ occurs in $x|_S$ if and only if a block of the form $\binom{ B^{\qq} }{ C }$ appears in $\binom x{y}$, where $C$ is any 0-1-block which starts and ends with a $1$ and has in total $n_k$ symbols $1$,  
while $\qq=(q_1,q_2,\dots,q_{k-1})\succcurlyeq\pp$ is uniquely determined by $C$ according to the following rule: 
\[
q_i=m_{i+1}-m_i-1 \qquad (i=1,2,\dots,k-1),
\]
where $m_i$ is the position of the ${n_i}^\text{th}$ symbol $1$ in $C$. 
To illustrate this dependence, suppose that $B=(b_1,b_2,b_3)$ and $\pp_0=(1,2)$. Then $B^{\pp_0}=(b_1,*,b_2,*,*,b_3)$ and $n_1=1$, $n_2=3$, $n_3=6$. As an example of a block
that starts and ends with $1$ and has $n_3=6$ symbols~$1$ we take $C=(1,0,0,1,1,0,1,1,0,1)$. The positions of the first, third and sixth symbol $1$ in $C$ are $m_1=1$, $m_2=5$, and $m_3=10$, respectively, which yields
$q_1=3$ and $q_2=4$. The resulting double block for this example is thus
\[
\binom{ B^{\qq} }{ C } = \begin{pmatrix}b_1,\!\!\!&*,\!\!\!&*,\!\!\!&*,\!\!\!&b_2,\!\!\!&*,\!\!\!&*,\!\!\!&*,\!\!\!&*,\!\!\!&b_3\\
1,\!\!\!&0,\!\!\!&0,\!\!\!&1,\!\!\!&1,\!\!\!&0,\!\!\!&1,\!\!\!&1,\!\!\!&0,\!\!\!&1 \end{pmatrix}.
\] 
The top row restricted to the positions of $1$'s in the bottom row matches $B^{\pp_0}$.

\smallskip
Although the assignment $C\mapsto\qq$ is not injective, it is surjective onto the set $\{\qq:\qq\succcurlyeq\pp_0\}$, and the vector $\qq=\pp_0$ is associated uniquely to the block $C_0$ consisting of $n_k$ consecutive symbols~$1$ (and no symbols $0$).

Exactly as in the proof of \Cref{distcf}, it can be shown that the lower frequency of $B^{\pp_0}$ in $x|_S$ does not exceed
\[
\sum_C \mu([B^{\bar q}])\nu([C]|[1]),
\]
where $C$ ranges over all 0-1-blocks which start and end with a~1 and have $n_k$ symbols~1 in total, while, in every summand, $\qq\succcurlyeq\pp_0$ denotes the vector associated to $C$ in the manner described above. The conditioning is with respect to the event that the first digit in the second row is 1. Because $\nu([1])>0$, the conditional probability (given that the first symbol in the second row is~1) of the event that there are jointly only finitely many symbols~1 in the second row is zero. Thus, with conditional probability one, to the right of the initial symbol~1 there appear infinitely many~1's, and hence there exists a block $C$ that starts at the initial~1, has $n_k$ symbols $1$, and ends with~1. This implies that the coefficients $\nu([C]|[1])$ sum up to~one. In this manner, we have estimated the lower frequency of $B^{\pp_0}$ in $x|_S$ by a convex combination of the numbers $\mu([B^\qq])$ where, by the preceding lemma, all of these numbers are strictly less than $\mu([B^{\pp_0}])$ except when $\qq=\pp_0$. This vector corresponds uniquely to the block $C_0$ consisting of just the symbols~1. If $\nu([C_0]|[1])=1$ then $\nu$-almost every element starting with~1 starts with $n_k$ consecutive symbols~$1$, in particular, the first digit~1 is followed by another digit $1$. In the preceding proof we have already shown that this is possible only when $\nu\in\mathsf{conv}\{\delta_{\bar 0},\delta_{\bar 1}\}$, an option which we have excluded. So $\nu([C_0]|[1])<1$ and the considered convex combination (and hence the lower frequency of $B^{\pp_0}$ in $x|_S$) is strictly less than $\mu([B^{\pp_0}])$. We have shown that $x|_S$ is not $\mu$-normal, as desired.  
\end{proof}

\begin{rem}
We can now see that the property of Bernoulli shifts which ``helps'' in $\mu_P$-normality preservation of deterministic sets is this:
for all $k\in\N$, all $B\in\Lambda^k$, and all vectors $\pp\in \N_0^{k-1}$, we have
\[
\mu_P([B])=\mu_P([B^{\pp}]).
\]
This property is called \emph{spreadability}. \Cref{lem spread block} shows that systems with CPE that are not Bernoulli shifts not only fail spreadability but fail it in a special way. Currently, there exist several proofs that the only spreadable systems are mixtures of Bernoulli shifts (i.e., systems in which almost all ergodic components are Bernoulli shifts); according to our search, the first such proof is due to Ryll-Nardzewski~\cite{Ry}. Nonetheless, spreadability is not necessary for $\mu$-normality preservation of selected deterministic sets. Clearly, any example of this kind must not involve a system with CPE. Two such examples are \Cref{garcia1,garcia2} from \Cref{sec examples}. We also give examples where both separability and disjointness are violated.
\end{rem}

To conclude our discussion of systems with CPE, we show that if $\mu$ is a measure with CPE then non-deterministic sets do not preserve $\mu$-normality. 

\begin{proof}[Proof of \Cref{main2}.] 
If $\mu$ is a Bernoulli measure $\mu_P$, then we can use the result \cite[Theorem~3]{Ka} of Kamae to see that non-deterministic sets do not preserve simple $\mu_P$-normality (let alone $\mu_P$-normality). So, our only concern are systems with CPE that are not Bernoulli shifts.

Let $x = (x_1,x_2,...) \in \Lambda^\N$ be $\mu$-normal for a non-Bernoulli CPE measure $\mu$. Let $S=\{s_1,s_2,\dots\}$ be a non-deterministic set. There exists a sequence $(N_k)_{k\ge1}$ along which $y=\mathbbm 1_S$ quasi-generates a measure $\nu$ of positive entropy. Note that this implies $\nu([1])\in(0,1)$. Since $\mu$ has CPE and hence is ergodic, it follows from \cite[Theorem 2]{Ka} that there exists a $\mu$-normal element $x\in\Lambda^\N$ such that the double \sq\ 
\[
\binom x{y} = \begin{pmatrix}x_1, & x_2, & x_3, & \dots \\ \1_S(1), & \1_S(2), & \1_S(3),& \dots \end{pmatrix},
\] 
quasi-generates (along $N_k$) the product joining $\mu\times\nu$.

From here the proof is identical as in the case of a non-superficial deterministic set $S$. We use the block $B^{\pp_0}$ from \Cref{lem spread block} and we estimate the lower frequency of $B^{\pp_0}$ in $x|_S$ by a convex combination of the numbers $\mu([B^\qq])$ all of which but one being strictly less than $\mu([B^{\pp_0}])$. The exception occurs when $\qq=\pp_0$, but the coefficient by which the corresponding term is multiplied is strictly smaller than one. This implies that the lower frequency of $B^{\pp_0}$ in $x|_S$ is less than $\mu([B^{\pp_0}])$ implying that $x|_S$ is not $\mu$-normal. This suffices to conclude that $S$ does not preserve $\mu$-normality.
\end{proof}

\pagebreak[3]
\section{Preservation of simple normality} \label{sec simple normality}

Recall that a set $S\subset\N$ ``preserves simple $\mu$-normality'' if for every $\mu$-normal element $x$ the element $x|_S$ is simply $\mu$-normal. The property of being simple $\mu$-normality preserving is clearly weaker than being (``full'') $\mu$-normality preserving. It is even possible that a set $S$ which preserves simple $\mu$-normality destroys $\mu$-normality. Indeed, our \Cref{main3} implies that deterministic sets of positive lower density preserve simple $\mu$-normality for measures $\mu$ with CPE, while, as soon as $\mu$ is not Bernoulli and $S$ is non-superficial, \Cref{main1} implies that $S$ destroys $\mu$-normality.
 
\begin{proof}[Proof of \Cref{main3}] The idea behind this proof is similar to that of~\cite[Theorem~5.1 (1)\,$\Rightarrow$\,(2)]{BDV}. Suppose that a set $S=\{s_1,s_2,\dots\}$ of positive lower density does not preserve simple $\mu$-normality. Then there exists a $\mu$-normal element $x\in\Lambda^\N$ such that for some symbol $b\in\Lambda$~the limit
\[
\lim_{N\to\infty}\frac{\#\{i\in\N: s_i\le N\text{ and }x_{s_i}=b\}}{\#\{i\in\N: s_i\le N\}}
\]
either does not exist or is different from $\mu([b])$. In either case there exists an increasing \sq\ $(N_k)_{k\ge1}$ along which the above limit exists and is different from $\mu([b])$. We can also assume that along the same \sq\ the double \sq\ $\binom{x}{y}$ (where, as usually, $y=\1_S$) quasi-generates a joining $\xi$ of $\mu$ with some shift \im\ $\nu$ derived from $S$. Since $S$ has positive lower density, we have $\nu([1])>0$. We can thus write
\[
\mu([b])\neq \lim_{k\to\infty}\frac{\#\{i\in\N: s_i\le N_k\text{ and }x_{s_i}=b\}}{N_k}\frac{N_k}{\#\{i\in\N: s_i\le N_k\}}=\frac{\xi\big(\doublecyl a1\big)}{\nu([1])}.
\]
We have shown that $\xi\big(\doublecyl b1\big)\neq\mu([b])\nu([1])$. Therefore $\xi$ is not the product joining of $\mu$ and~$\nu$, and hence $\mu$ is not disjoint from the measure $\nu$ derived from $S$. The contrapositive of \Cref{main3} (and therefore \Cref{main3} itself) is thus proved.
\end{proof}

The converse to \Cref{main3} need not hold; below we give three examples (\ref{morse1}, \ref{morse2}, and~\ref{trivex1}) of this phenomenon, in which $\mu$-normality preservation (and hence simple $\mu$-normality preservation), holds without disjointness. It remains an open problem whether the lack of disjointness resulting from $\mu$ and $\nu$ both having positive entropy is strong enough to exclude simple $\mu$-normality preservation. See \Cref{questions} for more details concerning this problem.

\section{Examples} \label{sec examples}
The results proved above may suggest that the phenomenon of $\mu$-normality preservation exists exclusively for Bernoulli shifts. Notice, however, that we claim so only within the class of systems with CPE. Beyond this class the phenomenon may still occur and with various configurations of determinism of $S$, entropy of $\mu$, and disjointness. To illustrate this, we now provide six examples, summarized in \Cref{fig examples}. In all these examples the sets $S$ are non-superficial.

\begin{figure}[htb]
    \begin{tabular}{|lll|l|} \cline{1-3}
    $S$ deterministic & entropy of main measure & disjointness \\ \hline
    yes & zero & yes & \Cref{garcia1} \\
    yes & zero & no & \Cref{morse1}  \\
    yes & positive & yes & \Cref{garcia2} \\
    yes & positive & no & \Cref{morse2}  \\
    no & zero &yes & \Cref{trivex} \\
    no & zero &no & \Cref{trivex1} \\
    no & positive & yes & impossible \\
    no & positive & no & unknown \\ \hline
    \end{tabular}
    \caption{Summary of examples of $\mu$-normality preservation.}
    \label{fig examples}
\end{figure}

Our first example consists of a shift system of entropy zero and an arithmetic progression (hence deterministic) $S$ whose derived measure is disjoint from $\mu$ and which preserves $\mu$-normality.
\begin{example}\label{garcia1}
Let $x\in\{0,1\}^\N$ be the ``Garcia--Hedlund sequence'' (first appearing in \cite{GH}), in which $x_n$ is $1$ if and only if the highest power of $2$ that divides $n$ is even. Alternatively, $x$ can be constructed from the following process~(see \cite[Lemma 2]{Aetal}).
\begin{enumerate}[\quad{Step }1:]
    \item put 1 at every odd position.
    \item put 0 at every position congruent to $2\!\!\mod 4$.
    \item put 1 at every position congruent to $4\!\!\mod 8$.
    \item put 0 at every position congruent to $8\!\!\mod 16$.
    \item[etc.] (use alternately 1's in odd steps and 0's in even steps.)
\end{enumerate}
This process defines $x$ at all coordinates. Let $S=3\N$. Explicitly,
\begin{align*}
          x&=101110101011101110111010101110101011101010111011101\dots,\\
y=\mathbbm1_S&=001001001001001001001001001001001001001001001001001\dots.
\end{align*}
The shift-orbit closure $X$ of $x$ is a so-called ``{regular Toeplitz system}'', which carries a unique \im~$\mu$ isomorphic to the unique \im\ of the dyadic odometer. So, we have constructed a measure-preserving shift system $(X,\mu,\sigma)$. As a consequence of unique ergodicity, $X$ satisfies \emph{uniform $\mu$-normality}: for any $k\in\N$ and any $\varepsilon>0$ there exists $n_0\in\N$ such that for any $n\ge n_0$, any $\ell\in \N$, any block $B\in\Lambda^k$, and any $y\in X$, we have
\[
\big|\mathsf{Fr}_{(y_\ell,y_{\ell+1},...,y_{\ell+n})}(B) - \mu([B])\big| < \varepsilon.
\]

\medskip
We claim that $S$ preserves $\mu$-normality. Observe that $y$ is generic for the periodic measure $\nu$ supported by a cycle of three points, with atoms of masses $\frac13$ each. It is well known that the dyadic odometer is disjoint from any periodic system with an odd period, so $\mu$ and $\nu$ are disjoint. 

First, we analyze $x|_S=(x_3,x_6,x_9,x_{12},x_{15},\dots)$. Note that all positions that are odd relatively in $x|_S$ come from odd positions in $x$, so $x|_S$ inherits at its odd places the $1$'s inserted in~$x$ in Step~1. Likewise, the positions that are congruent to $2\!\!\mod 4$ in $x|_S$ come from the positions congruent to $2\!\!\mod 4$ in~$x$. So $x|_S$ inherits at these places the $0$'s inserted in $x$ in Step~2. Arguing in this manner we come to the conclusion that $x|_S$ actually matches $x$. In this most obvious way,~$S$ preserves $\mu$-normality of~$x$.

Next, we observe $x|_{S+1}=(x_4,x_7,x_{10},x_{13},x_{16},\dots)$. This time every even position in $x|_{S+1}$ comes from an odd position in $x$ (so inherits a $1$ from the first step), every position congruent to $3\!\!\mod 4$ in $x|_{S+1}$ comes from a position congruent to $2\!\!\mod 4$ in $x$ (so inherits a $0$ from the second step), and so on. We deduce that $x|_{S+1}$ has the same structure as $x$ with the periodic parts filled in the Steps~1, 2, 3, etc.~shifted. It is an elementary property of regular Toeplitz systems that such an element belongs to the orbit closure of $X$ and hence is $\mu$-normal. Analogously, $x|_{S+2}$ belongs to $X$ and is $\mu$-normal. 

Finally, take any element $z\in\Lambda^\N$ that is $\mu$-normal (it need not even belong to $X$). Fix $k\in\N$ and $\varepsilon>0$ and let $n_0$ be as in the uniform $\mu$-normality condition. It follows from $\mu$-normality of $z$ that after ignoring a set of coordinates of density zero, $z$ consists of blocks $C$ longer than $3n_0$ appearing in~$x$. Thus, after ignoring a set of coordinates of density zero, $z|_S$ consists of blocks of the form $C|_S$ longer than $n_0$ that appear in either $x|_S$ or in $x|_{S+1}$ or in $x|_{S+2}$ (depending on the congruence modulo 3 of the starting position of $C$ in $y$). In any case, all these blocks appear in some elements of $X$. By uniform $\mu$-normality of $X$, $z|_S$ satisfies the condition
\[
|\mathsf{Fr}_{z|_S}(B) - \mu([B])|<\varepsilon
\]
for any block $B\in\Lambda^k$. Since this is true for any $k\in\N$ and any $\varepsilon>0$, we obtain that~$z|_S$ is~$\mu$-normal. Thus we have shown that~$S$ preserves $\mu$-normality.
\end{example}

In the next example $\mu$ has entropy zero, $S$ is non-superficial, deterministic, and has positive lower density, and its derived measure~$\nu$ is \emph{not} disjoint from $\mu$ (in fact, it is a factor of $\mu$), and yet $S$ preserves $\mu$-normality.

\begin{example}\label{morse1}
Let now $x\in\{0,1\}^\N$ denote the classical Thue--Morse \sq\footnote{Although the sequence appears earlier in works of Axel Thue, its popularity stems from a 1921 paper of Harold Marston Morse~\cite{Morse}.} (one definition is that $x_n = 1$ if and only if the number of $1$'s in the binary expansion of $n-1$ is odd), and let $S=2\N+1$ be the set of odd natural numbers. Explicitly,
\begin{align*}
x &= 01101001100101101001011001101001\dots,\\
y=\mathbbm1_S&=10101010101010101010101010101010\dots.
\end{align*}
Let $X$ denote the orbit closure of $x$.\footnote{It is well known (see, for example, \cite{Aetal}) that the map $\phi$ defined by $(\phi(z))_n = z_n+z_{n+1} \bmod 2$ maps the Morse system $X$ onto the Toeplitz system of the preceding example. In fact, $\phi$ is 2-to-1, sending both $x$ and its negation to the Garcia--Hedlund \sq.} The system $(X,\sigma)$ is known to be minimal and uniquely ergodic. We let $\mu$ denote the unique shift-\im\ on $X$. The unique measure $\nu$ derived from $S$ is supported by a cycle of two points. Since $X$ factors onto the dyadic odometer which, in turn, factors onto the two-cycle, $\nu$ is a factor of $\mu$. Nevertheless, it is almost immediate to see that $x|_S$ actually matches $x$, while $x|_{S+1}$ matches the negation of $x$ (which belongs to $X$ as well). So, both $x|_S$ and $x|_{S+1}$ are $\mu$-normal. From here we argue in the same way as in \Cref{garcia1} and deduce that $S$ preserves $\mu$-normality.
\end{example}

The following is a modification of \Cref{garcia1} so that the measure $\mu'$ has positive entropy.
\begin{example}\label{garcia2}
Let $\mu'=\mu\times\mu_P$ be the product of the measure $\mu$ from \Cref{garcia1} with the uniform Bernoulli measure on two symbols. Clearly, $\mu'$ has positive entropy without being a Bernoulli shift (and it does not have CPE, which makes this example possible). Let $S=3\N$, as in \Cref{garcia1}. The measure $\nu$ derived from $S$ is disjoint from $\mu'$ because it is disjoint from both $\mu$ and $\mu_P$. We claim that $S$ preserves $\mu'$-normality. Since $\mu$ and $\mu_P$ are disjoint, a double \sq\footnote{This is not to be confused with the double sequence~$\binom x{y}$ considered earlier. Here, the double sequence results from the measure $\mu'$ being $\mu \times \mu_P$, so the top row relates to the Toeplitz system and the bottom row relates to the Bernoulli shift.}~$\binom zw$ is $\mu'$-normal if and only if $z$ is $\mu$-normal and $w$ is $\mu_P$-normal (i.e., classically normal). Then
the sub\sq\
\[
\left.\binom zw\right|_S 
\]
has in the first row the \sq\ $z|_S$, which is $\mu$-normal by \Cref{garcia1}, and in the second row the \sq\ $w|_S$, which is classically normal by Wall's Theorem. So, $\left.\binom zw\right|_S$ is $\mu'$-normal, as claimed. 
\end{example}

The next example is similar to \Cref{morse1}, except that now our main measure has positive entropy. 
\begin{example}\label{morse2}
Let $\mu'=\mu\times\mu_P$ be the product of the Thue--Morse measure $\mu$ with the uniform Bernoulli measure $\mu_P$ on two symbols. Clearly, $\mu'$ has positive entropy.  Let $S$ be the set of odd numbers once again.  Because, as explained in \Cref{morse1}, $\mu$ is not disjoint from $\nu$, also $\mu'$ is not disjoint from $\nu$. However, by \Cref{morse1}, $S$ preserves $\mu$-normality. From here, the argument showing that $S$ preserves $\mu'$-normality is identical to the argument in \Cref{garcia2}.
\end{example}

The final two examples concern non-deterministic sets that preserve normality with respect to zero-entropy measures.

\begin{example}\label{trivex}
Let $\mu$ be the measure supported by the orbit closure of the 2-periodic \sq\ $x=010101\dots$. Let $s_0=0$ and let $S=\{s_n:n\ge 1\}$ be defined inductively by $s_n=s_{n-1}+\omega_n$, where $(\omega_n)\in\{1,3\}^\N$ is a fixed \sq\ that is generic for the uniform Bernoulli measure on two symbols. In other words, the distances between consecutive elements of $S$ are either $1$ or $3$ both appearing ``at random''. 
For example, $S$ might correspond to the \sq
\[
y=\mathbbm1_S=11001001111001001001100100111001001001001100111111001001\dots.
\]
It is elementary to see that $S$ has density $\frac12$ and generates a measure $\nu$ isomorphic to the mixing Markov process on three symbols given by the transition matrix
\[ \begin{bmatrix}
    \frac12 & \frac12 & 0\\
    0 & 0 & 1\\
    1 & 0 & 0
\end{bmatrix}. \]
Every mixing Markov process is isomorphic to a Bernoulli process, so it has CPE and hence $\nu$ and $\mu$ are disjoint.
Every $\mu$-normal element $z$ consists, up to density zero, of very long blocks appearing in $x$, i.e., blocks of the form $0101\dots01$. The set $S$ consists of alternating odd and even numbers in bounded distances, which implies that $z|_S$ consists (up to density zero) of long blocks of the form $0101\dots01$ as well. Thus $z|_S$ is $\mu$-normal and hence $S$ preserves $\mu$-normality.
\end{example}

\begin{example}\label{trivex1}
Let $\mu$ be generated by the $4$-periodic \sq\ $x=00110011001100110011\dots$. Define $S=\{s_1,s_2,\dots\}$ by letting $s_1=1$ and then, for $n\ge2$, $s_n = s_{n-1}+\omega_n$, where $\omega_{2k}=2$ for all $k\in\N$ while $\omega_{2k-1}=4$ or $8$ with probabilities~$\frac12$ each.
This time, $S$ generates a measure $\nu$ isomorphic to a Markov process (of positive entropy) that is not mixing and factors to a two-cycle (any element of the orbit closure of $y=\mathbbm1_S$ has either all symbols $1$ at even positions or all symbols $1$ at odd positions). Clearly, $\mu$ also factors to a two-cycle, so $\mu$ and $\nu$ are not disjoint. To see that $S$ preserves $\mu$-normality, observe that by moving $4$ or $8$ positions forward we never change the symbol in $x$, while by moving $2$ positions forward we always do. So, $x|_S=x$. Then we apply the already-familiar argument about any $\mu$-normal element being practically a concatenation of long blocks appearing in $x$ to see that $S$ preserves $\mu$-normality.
\end{example}

\section{Open questions}\label{questions}
These examples show that, unless $\mu$ has CPE, disjointness (of $\mu$ from the measures derived from $S$) is not necessary for $\mu$-normality preservation of $S$. In general, disjointness implies simple $\mu$-normality preservation, but not the other way around. Also, stretchability of $\mu$ is not necessary, regardless of whether the disjointness condition is satisfied or not. Moreover, any restrictions on the entropy of $\mu$ and determinism of $S$ seem to be irrelevant.

Hence the following questions are a mystery:
\begin{que}
Let $\mu$ be any ergodic shift-\im. What necessary condition must a set $S$ satisfy to be (simple) $\mu$-normality preserving? (We mean here nontrivial conditions other than those in \Cref{jeden}.)
\end{que}

\begin{que}
Is there a ``checkable'' sufficient condition for the pair (measure, set), $(\mu,S)$, which guarantees that $S$ preserves $\mu$-normality? 
\end{que}

A more particular question concerns the only configuration of parameters for which we have no suitable example:

\begin{que} \label{former claim}
Does there exist a measure $\mu$ with positive entropy and a non-deterministic set $S$ such that $S$ preserves simple $\mu$-normality? Is this possible for some $\mu$ with CPE? (By Kamae~\cite{Ka}, this is known to be impossible if $\mu$ is i.i.d.)
\end{que}

A negative answer to \Cref{former claim} would follow from a positive answer to the following question (that is, from a strengthening of \cite[Theorem~5.5]{BDV}):
\begin{que} \label{like5.5}
Suppose $(\Lambda_1^\N,\mu,\sigma)$ and $(\Lambda_2^\N,\nu,\sigma)$ are arbitrary shift systems on finite or countable alphabets, each with positive entropy. Must there exist a joining $\xi = \mu \vee \nu$ that makes the~``zero-coordinate partitions'' $\{[a]:a\in\Lambda_1\}$ and $\{[b]:b\in\Lambda_2\}$ dependent?
\end{que}
This last question is nontrivial already in the class of non-i.i.d.~Bernoulli systems.

\bigskip
\subsection*{Acknowledgements} The authors would like to thank Vitaly Bergelson for suggesting the topic. We also thank the referee of the first version of the paper for indicating a flaw which lead us to the discovery of the phenomena associated with superficial sets. The research of the second author is supported by NCN grant 2018/30/M/ST1/00061.

\providecommand\doi[1]{}

\end{document}